\documentclass[a4paper,reqno,12pt]{amsart}

\usepackage{amsmath,amssymb,mathtools,bm,bbm} 
\usepackage[utf8]{inputenc} 
\usepackage{url}
\usepackage{tikz,tikz-cd}
\usetikzlibrary{decorations.pathmorphing} 
\usetikzlibrary{calc}
\usepackage{enumitem}
\usepackage{color}

\usepackage{hyperref}
\hypersetup{
    colorlinks,
    linkcolor={blue},
    citecolor={blue},
    urlcolor={blue!80!black}
}

\newtheorem{theorem}{Theorem}[section]
\newtheorem{lemma}[theorem]{Lemma}
\newtheorem{corollary}[theorem]{Corollary}
\newtheorem{proposition}[theorem]{Proposition}
\newtheorem{remark}[theorem]{Remark}

\theoremstyle{definition}
\newtheorem{definition}[theorem]{Definition}
\newtheorem{example}[theorem]{Example}

\newcommand{\lrvert}[1]{{\left\vert{#1}\right\vert}} 
\newcommand{\lrangle}[1]{\langle #1 \rangle} 
\newcommand{\bsm}[1]{\begin{bsmallmatrix} #1 \end{bsmallmatrix}} 

\newcommand{\lcm}{\operatorname{lcm}} 

\newcommand{\bZ}{\mathbb{Z}} 

\newcommand{\bP}{\mathbb{P}} 
\newcommand{\bL}{\mathbb{L}} 
\newcommand{\bX}{\mathbb{X}} 
\newcommand{\coh}{\operatorname{coh}} 

\renewcommand{\dim}{\operatorname{dim}} 
\renewcommand{\mod}{\operatorname{mod}} 
\newcommand{\K}{\operatorname{\mathrm{K}}} 
\newcommand{\Der}{\operatorname{\mathrm{D}}} 
\newcommand{\perf}{\operatorname{perf}} 
\newcommand{\lperf}{\operatorname{lperf}} 
\newcommand{\rHom}{\operatorname{\mathbf{R}\!\Hom}} 
\newcommand{\lotimes}{\operatorname{\overset{\mathbf{L}}{\otimes}}} 
\newcommand{\Hom}{\operatorname{Hom}} 
\newcommand{\End}{\operatorname{End}} 
\newcommand{\Cone}{\operatorname{Cone}} 
\newcommand{\cx}{\operatorname{cx}} 
\newcommand{\glcx}{\operatorname{gl.cx}} 

\newcommand{\op}{\textup{op}} 

\begin{document}

\title{Serre Cyclotomic Algebras}

\author{Calvin Pfeifer}
\address{Mathematical Institute, University of Cologne, Weyertal 86-90, 50931 Köln, Germany}
\email{cpfeifer@math.uni-koeln.de}

\date{\today}

\subjclass[2020]{}

\begin{abstract}
    We introduce a class of proper differential graded algebras which we call \emph{Serre cyclotomic}.
    They generalize fractionally Calabi–Yau algebras and categorify de la Peña's algebras of cyclotomic type. 
    Path algebras of affine type and (higher) canonical algebras are examples of Serre cyclotomic algebras.
    Their definition is related to Elias--Hogancamp's theory of categorical diagonalization. 
    We compute the categorical entropy of their Serre functors, in the sense of Dimitrov--Haiden--Katzarkov--Kontsevich, 
    and use this to determine which graded path algebras and which homologically smooth graded gentle algebras are Serre cyclotomic.
    Finally, we show that trivial extension algebras of Serre cyclotomic algebras have finite global complexity.
\end{abstract}

\maketitle

\setcounter{tocdepth}{2}
\tableofcontents


\section*{Introduction}

The trichotomy of quivers into finite, affine and indefinite type, 
according to the definiteness of their associated Tits quadratic form,
has various counterparts in e.g. representation theory, singularity theory, Lie theory, and cluster algebra theory.
The quivers of finite type are precisely those quivers $Q$ whose Coxeter transformation $\Phi_Q$ is periodic.
This may be seen as a shadow of the following representation theoretic characterization, c.f. \cite{MY}: 
A quiver $Q$ is of finite type if and only if its path algebra $A = KQ$ is \emph{fractionally Calabi--Yau},
i.e. the Serre functor $\boldsymbol{\nu}_A$ on the perfect derived category $\perf(A)$ is periodic up to a shift.
More precisely, there exist $m,n\in \mathbb{Z}$ with $n\neq 0$ and for every $V\in\perf(A)$ an isomorphism $\boldsymbol{\nu}_A^n(V) \cong V[m]$ which is functorial in $V$.
The Coxeter transformation $\Phi_Q = \Phi_A$ is the map on the Grothendieck group $\K_0(A)$ of $\perf(A)$
induced by the derived Auslander--Reiten translation $\boldsymbol{\tau}_A := \boldsymbol{\nu}_A[-1]$.
In particular, the Coxeter transformation of any fractionally Calabi--Yau algebra is periodic.
Other well-known examples of fractionally Calabi--Yau algebras are symmetric algebras and canonical algebras of tubular type. 
Further examples are some higher representation-finite algebras \cite{HI} \cite{Gr} \cite{Ha},
and incidence algebras of Tamari lattices \cite{Ro}.

Quivers of affine type can also be characterized in terms of their Coxeter matrices:
A quiver $Q$ is of affine type if and only if its Coxeter transformation $\Phi_Q$ is 
cyclotomic\footnote{We call $\Phi_Q$ cyclotomic if some non-zero power of $\Phi_Q$ is unipotent.} 
but not periodic.
More generally, de la Peña \cite{dlP} defines a triangular finite-dimensional algebra $A$ to be of \emph{cyclotomic type} if its Coxeter transformation $\Phi_A$ is cyclotomic.
Path algebras $A=KQ$ of affine quivers $Q$ are not fractionally Calabi--Yau but in some sense 
they are only one step away:
There exists an integer $n\geq 1$ and for every $V\in \perf(A)$ there exists a homomorphism $f\colon V \to \boldsymbol{\tau}_A^{-n}(V)$ such that $\boldsymbol{\tau}_A^{-n}(\Cone(f)) \cong \Cone(f)$; 
moreover, both the morphism $f$ and the isomorphism are natural.
In Section \ref{sec:serre_cyclotomic}, we formalize this in the notion of \emph{Serre cyclotomic algebras}.
It follows that the Coxeter transformation of any Serre cyclotomic algebra is cyclotomic.
In Section \ref{sec:canonical} we will see that higher canonical algebras are natural examples of Serre cyclotomic algebras.
Higher canonical algebras were introduced in \cite{HIMO},
they generalize Ringel's canonical algebras 
and are derived equivalent to Geigle--Lenzing complete intersections.
From this noncommutative geometry perspective, Serre cyclotomicity corresponds to the well-known fact that Geigle--Lenzing complete intersections are Fano, Calabi--Yau, or anti-Fano.
As a noncommutative analog of Fano varieties, Minamoto \cite{Mi} introduced the notion of Fano algebras. 
While path algebras of all quivers of infinite type are Fano algebras, those of indefinite type are not Serre cyclotomic.

In this note, we work more generally with proper differential graded algebras and differential graded bimodules.
In Section \ref{sec:cyclotomic_bimodules}, we introduce \emph{cyclotomic dg bimodules}. 
We relate this notion to Elias--Hogancamp's categorical eigentheory \cite{EH},
and study their categorical dynamics in the sense of \cite{DHKK}.
It follows from these considerations that graded path algebras and homologically smooth graded gentle algebras cannot be Serre cyclotomic unless they are already derived equivalent to path algebras of type $\mathbb{A}$ or $\tilde{\mathbb{A}}$ concentrated in degree $0$;
see Sections \ref{sec:graded_quiver} and \ref{sec:graded_gentle}. 
In the final Section \ref{sec:complexity}, 
we use Happel's equivalence to show that the trivial extension algebra $TA := A \ltimes DA$,
of any Serre cyclotomic algebra $A$ which is concentrated in degree $0$ and has finite global dimension,
has finite complexity, in the sense of Alperin \cite{Al}.


\subsection*{Setting}

We fix an algebraically closed field $K$.
For convenience, we abbreviate ``\emph{differential graded}'' by ``\emph{dg}'' in what follows.
Let $A$ be a proper dg $K$-algebra.
All our dg $A$-modules are left dg $A$-modules.
We write $\Der(A)$ for the derived category of dg $A$-modules, 
and $\perf(A)$ for the smallest thick subcategory of $\Der(A)$ containing $A$.
Throughout, we abbreviate $\Hom_A := \Hom_{\Der(A)}$.
Note that the opposite algebra $A^\op$
and its enveloping algebra $A^e := A \otimes_K A^\op$
are both proper again.
We canonically identify dg $A^e$-modules with dg $A$-bimodules.


\section{Cyclotomic bimodules} \label{sec:cyclotomic_bimodules}

Let $A$ be a proper dg algebra.
Every $M\in \Der(A^e)$ defines a pair of adjoint functors 
\begin{equation}\label{adjunction}
    \begin{tikzcd}
        M \lotimes_A ? \colon \Der(A) \ar[r,shift left = 1mm] & \Der(A) \,\colon\! \rHom_A(M,?) \ar[l,shift left = 1mm] \,.
    \end{tikzcd}
\end{equation}
Note that the left adjoint restricts to a well-defined functor $M\lotimes_A - \colon \perf(A) \to \perf(A)$
if and only if $M\in \perf(A)$ as a left dg $A$-module. 
Such dg $A$-bimodules form a thick subcategory which we denote by
\begin{align*}
    \lperf(A^e) &:= \{M\in \Der(A^e) \mid \text{$M \in \perf(A)$ as a left dg $A$-module}\} \subseteq \Der(A^e).
\end{align*}
For $M\in \lperf(A^e)$ and $\tilde{M} := \rHom_A(M,A) \in \Der(A^e)$, we have a natural isomorphism
\begin{align*}
    \rHom_{A}(M,-) \simeq \tilde{M}\lotimes_A -
\end{align*}
We call $M$ \emph{invertible} if there are quasi-isomorphisms $M\lotimes_A \tilde{M} \simeq A \simeq \tilde{M}\lotimes_A M$ in $\Der(A^e)$.
In this case $\tilde{M}\in \lperf(A^e)$ and (\ref{adjunction}) is an equivalence.
For $M\in\lperf(A^e)$ and $n\geq 1$, we abbreviate
\begin{align*}
    M^n := \underbrace{M \lotimes_A \cdots \lotimes_A M}_{\textup{$n$ times}} && \text{and} &&
    M^{-n} := \underbrace{\tilde{M} \lotimes_A \cdots \lotimes_A \tilde{M}}_{\textup{$n$ times}}\,.
\end{align*}

\begin{definition}\label{def:cyclotomic_bimodule}
    Let $l,m,n\in \bZ$ with $l\geq 1$ and $n\neq 0$.
    We call a dg bimodule $M\in \lperf(A^e)$ \emph{cyclotomic} of \emph{dimension} $(m,n)$ and \emph{order} at most $l$ 
    if $M$ is invertible and there exist homomorphisms $f_1,\dots,f_l\in\Hom_{A^e}(A,M^n[m])$ such that 
    \begin{equation*}
        \Cone(f_l) \lotimes_A \cdots \lotimes_A \Cone(f_1) \simeq 0 \in \lperf(A^e)\,.
    \end{equation*}
    We sometimes say that $M$ is \emph{$(l,m,n)$-cyclotomic}.
\end{definition}

Denote by $\K_0(A)$ the Grothendieck group of the triangulated category $\perf(A)$.
Every left perfect dg $A$-bimodule $M\in \lperf(A^e)$ induces a homomorphism 
\begin{align*}
    \Psi_M \colon \K_0(A) \rightarrow \K_0(A)\,, \hspace{1cm} [X] \mapsto [M\lotimes_A X]\,.
\end{align*}
As a direct consequence of the definition we obtain that cyclotomic dg bimodules induce cyclotomic maps on Grothendieck groups.

\begin{proposition}\label{prop:induced_map_cyclotomic}
    Let $l,m,n\in \bZ$ with $l\geq 1$ and $n\neq 0$.
    Suppose $M\in\lperf(A^e)$ is $(l,m,n)$-cyclotomic.
    Then, the induced map $\Psi_M$ is cyclotomic with 
    \begin{align*}
        (\Psi_M^{2n} - \bm{1})^l = 0\,.
    \end{align*}
\end{proposition}
\begin{proof}
    For $M,N\in\lperf(A^e)$, $L:=M\lotimes_A N$ and every $f\in \Hom_{A^e}(M,N)$ we have 
\begin{align*}
    \Psi_0 = 0, &&
    \Psi_A = \bm{1}, &&
    \Psi_{M[1]} = -\Psi_M, &&
    \Psi_{L} = \Psi_M \Psi_N, && \text{and} &&
    \Psi_{\Cone(f)} = \Psi_N - \Psi_M.
\end{align*}
    With these identities, the claim follows immediately from Definition \ref{def:cyclotomic_bimodule}.
\end{proof}

\subsection{Categorical eigentheory} \label{sec:categorical_eigentheory}

Elias--Hogancamp developed in \cite{EH} their theory of categorical diagonalization.
In this section, we generalize their definition of eigencategories to generalized eigencategories.
Note that our definition is different from their suggestion in \cite[Section 9.4]{EH}.

\begin{definition}\label{def:generalized_eigencategory}
    Let $M,N\in \lperf(A^e)$, $l\geq 1$ and ${\bf f} \in \Hom_{A^e}(M,N)^l$.
    For convenience, we abbreviate
    \begin{align*}
        \Cone({\bf f}) := \Cone(f_l) \lotimes_A \cdots \lotimes_A \Cone(f_1) \in \lperf(A^e).
    \end{align*}
    We call the following full subcategories of $\perf(A)$ \emph{generalized eigencategories}:
    \begin{align*}
        \mathcal{E}^{\bf f}(M,N) := \{X\in\perf(A) \mid \text{$\Cone({\bf f}) \lotimes_A X \simeq 0$}\} \subseteq \perf(A),
    \end{align*}
    \begin{align*}
        \mathcal{E}^l(M,N) := \bigcup_{{\bf f}\in \Hom_{A^e}(M,N)^l} \mathcal{E}^{\bf f}(M,N) \subseteq \perf(A).
    \end{align*}
\end{definition}

For $m\in \bZ$, $M=A[m]$, $l=1$ and $f\in \Hom_{A^e}(A[m],N)$, the category $\mathcal{E}^f(A[m],N)$ is the \emph{eigencategory} of $N$ with \emph{eigenvalue} $m$ and \emph{eigenmap} $f$ from \cite{EH}.
Cyclotomic dg bimodules can be characterized in terms of generalized eigencategories as follows.

\begin{lemma}\label{lem:cyclotomic_generalized_eigencategory}
    Let $l,m,n\in \mathbb{Z}$ with $l\geq 1$ and $n\neq 0$.
    The following are equivalent for an invertible $M\in \lperf(A^e)$:
    \begin{enumerate}[label = (\roman*)]
        \item $M$ is $(l,m,n)$-cyclotomic.
        \item $\mathcal{E}^l(A,M^n[m]) = \perf(A)$.
    \end{enumerate}
\end{lemma}
\begin{proof}
    This is immediate from Definitions \ref{def:cyclotomic_bimodule} and \ref{def:generalized_eigencategory}.
\end{proof}

In particular, if a dg bimodule $M$ is $(l,m,n)$-cyclotomic,
then $M^n$ is prediagonalizable in the sense of \cite{EH}.
The next two lemmas are sometimes useful for showing that a dg bimodule is Serre cyclotomic.

\begin{lemma}\label{lem:generalized_eigencategory_thick}
    Let $M,N\in\lperf(A^e)$ and $l\geq 1$.
    The full subcategory $\mathcal{E}^l(M,N)$ of $\perf(A)$ is a thick subcategory.
\end{lemma}
\begin{proof}
    For each ${\bf f}\in \Hom_{A^e}(M,N)^l$, it is clear that $\mathcal{E}^{\bf f}(M,N)\subseteq \perf(A)$ is a thick subcategory. 
    It follows immediately that $\mathcal{E}^l(M,N)$ is closed under direct summands and shifts,
    and for cones it thus suffices to show: For any two $X,Y\in \mathcal{E}^l(M,N)$ there exists some ${\bf f}\in \Hom_{A^e}(M,N)^l$ such that $X,Y\in \mathcal{E}^{\bf f}(M,N)$.
    
    The subset 
    \begin{align*}
        \mathcal{U}(X) := \{{\bf f} \mid \text{$\Cone({\bf f}) \lotimes_A X \simeq 0$ in $\perf(A)$}\}\subseteq \Hom_{A^e}(M,N)^l
    \end{align*}
    is Zariski open.
    Indeed,
    consider
    \begin{align*}
        X' := \Cone(f_{l-1})\lotimes_A \dots \lotimes_A \Cone(f_1) \lotimes_A X \in \perf(A),
    \end{align*}
    the linear map 
    \begin{align*}
        F:= - \lotimes_A X'\colon \Hom_{A^e}(M,N) \xrightarrow{} \Hom_{A}(M\lotimes_A X', N\lotimes_A X')
    \end{align*}
    and the Zariski open subset
    \begin{align*}
        \mathcal{U'} := \{h \mid \text{$h$ is a quasi-isomorphism}\} \subseteq \Hom_{A}(M\lotimes_A X', N\lotimes_A X').
    \end{align*}
    It follows that $\mathcal{U}(X) = F^{-1}(\mathcal{U}')$ is Zariski open in $\Hom_{A^e}(M,N)^l$.
    Since $K$ is infinite by our standing assumption,
    the affine space $\Hom_{A^e}(M,N)^l$ is irreducible.
    Thus,
    for any two $X,Y\in \mathcal{E}^l(M,N)$,
    the intersection $\mathcal{U}(X) \cap \mathcal{U}(Y)$ is non-empty.
    For any ${\bf f}$ in that intersection, 
    we have $X,Y \in \mathcal{E}^{\bf f}(M,N)$.
\end{proof}

\begin{lemma}\label{generalized_eigencategory_invariant}
    Let $M\in \lperf(A^e)$ be invertible and $l\geq 1$.
    The full subcategory $\mathcal{E}^l(A,M)$ is closed under the functors $M\lotimes_A -$ and $\tilde{M}\lotimes_A-$.
\end{lemma}
\begin{proof}
    Suppose $X\in\mathcal{E}^l(A,M)$.
    Set $X' := M\lotimes_A V$.
    We want to show $X' \in \mathcal{E}^l(A,M)$. 
    Let ${\bf f}\in \Hom_{A^e}(A,M)^l$
    with $\Cone({\bf f}) \lotimes_A X \simeq 0$.
    Define another sequence of morphisms
    ${\bf f'} = (f'_l,\dots,f'_1)\in \Hom_{A^e}(A,M)^l$ 
    by the following commutative diagram with triangle rows and vertical quasi-isomorphisms
    for each $1\leq i \leq l$:
    \[
    \begin{tikzcd}[sep=.8cm]
        A \ar[rr,"f_i'"] \ar[d,"\sim"] && M \ar[r] \ar[d,"\sim"] & \Cone(f_i') \ar[r] \ar[d,dashed,"\sim"] & A[1] \ar[d,"\sim"] \\
        M \lotimes_A A \lotimes_A \tilde{M} \ar[rr,"M \lotimes f_i \lotimes \tilde{M}"] && M \lotimes_A M \lotimes_A \tilde{M} \ar[r] & M \lotimes_A \Cone(f_i) \lotimes_A \tilde{M} \ar[r] &  ...[1]
    \end{tikzcd}
    \]
    Since $M$ is invertible, 
    we have $\tilde{M} \lotimes_A M \simeq A$ in $\lperf(A^e)$. 
    Thus, we find
    \begin{align*}
        \Cone({\bf f'}) &\lotimes_A (M\lotimes_A X) \simeq \\
        &\simeq M \lotimes_A \Cone(f_l) \lotimes_A \tilde{M} \lotimes_A \dots 
        \lotimes_A M \lotimes_A \Cone(f_1) \lotimes_A \tilde{M} \lotimes_A M \lotimes_A X \\
        &\simeq M \lotimes_A \Cone({\bf f}) \lotimes_A X \simeq 0.
    \end{align*}
    This shows $X'\in \mathcal{E}^l(A,M)$.
    Similarly, we see that $\mathcal{E}^l(A,M)$ is closed under $\tilde{M} \lotimes_A -$.
\end{proof}

\subsection{Categorical entropy}\label{sec:categorical_entropy}

Let $\mathcal{T}$ be a triangulated category.
Recall that $T\in\mathcal{T}$ is a generator
if for every $X\in\mathcal{T}$ there exists a sequence of triangles
\begin{equation}\label{eq:generator_sequence}
    \begin{tikzcd}[column sep = .1cm, row sep = .5cm]
        0 \ar[rr] && 
        X_1 \ar[rr] \ar[dl] && 
        X_2 \ar[rr] \ar[dl] && 
        \cdots \ar[rr] \ar[dl] && 
        X_{l-1} \ar[rr] \ar[dl] && 
        X \oplus X' \ar[dl] \\ &
        T[n_1] \ar[dotted,ul] && 
        T[n_2] \ar[dotted,ul] && 
        T[n_3] \ar[dotted,ul] & 
        \cdots & 
        T[n_{l-1}] \ar[dotted,ul] && 
        T[n_l] \ar[dotted,ul]
    \end{tikzcd}
\end{equation}
with $n_1,\dots,n_l\in\bZ$ and some $X'\in \mathcal{T}$.
In \cite{DHKK} the \emph{complexity} of an object $X\in\mathcal{T}$ at $t\in \mathbb{R}$
with respect to a generator $T\in\mathcal{T}$ is defined as
\begin{equation*}
    \delta_t(T,X) := \inf\left\{\sum_{i = 1}^l e^{n_i t} \mid \text{(\ref{eq:generator_sequence}) for $X$ exists}\right\} \in \mathbb{R}_{\geq 0}
\end{equation*}

Given an exact endofunctor $F\colon \mathcal{T} \to \mathcal{T}$ they define in \cite{DHKK} its \emph{entropy} at $t\in \mathbb{R}$ as
\begin{equation*}
    h_t(F) := \lim_{N\to \infty} \frac{1}{N} \log \delta_t(T,F^NT) \in \mathbb{R}\cup\{-\infty\},
\end{equation*}
which is well-defined and independent of the chosen generator $T\in\mathcal{T}$.
Moreover, Fan--Fu--Ouchi \cite{FFO} define the \emph{(lower) polynomial entropy} of $F$ as 
\begin{align*}
    h^{\textup{pol}}(F) := \liminf_{N\to\infty} \frac{\log(\delta_0(T,F^NT))-N\cdot h_0(F)}{\log(N)}\,.
\end{align*}

\begin{lemma}\label{lem:entropy_1}
    Let $\mathcal{T}$ be a triangulated category and $F\colon\mathcal{T} \to \mathcal{T}$ an exact endofunctor.
    Suppose there exist $X_0,\dots,X_l \in \mathcal{T}$ and for each $0\leq i < l$ a triangle
    \begin{equation}\label{eq:functor_filtration}
        \begin{tikzcd}[column sep = .25cm]
            X_i \ar[r] & FX_i \ar[r] & X_{i+1} \ar[r] & X_i[1] 
        \end{tikzcd}
    \end{equation}
    such that $X_0 = T$ and $X_l = 0$.
    Then $h_t(F) \leq 0$ for all $t\in\mathbb{R}$.
    If moreover $h_0(F) = 0$, then $h^{\textup{pol}}(F) \leq l-1$.
\end{lemma}
\begin{proof}
    Fix $t\in \mathbb{R}$.
    For any $N\geq 1$ and $0\leq i < l$, 
    apply $F^{N-1}$ to the triangle (\ref{eq:functor_filtration}) to obtain the triangle
    \begin{equation*}
        \begin{tikzcd}[column sep = .25cm]
            F^{N-1}X_i \ar[r] & F^{N}X_i \ar[r] & F^{N-1}X_{i+1} \ar[r] & F^{N-1} X_i[1]\,.
        \end{tikzcd}
    \end{equation*}
    From \cite[Proposition 2.3]{DHKK} we get 
    \begin{align*}
        \delta_t(T, F^N X_i) \leq \delta_t(T,F^{N-1}X_i) + \delta_t(F^{N-1}X_{i+1})
    \end{align*}
    For $N\geq l$,
    we find inductively that
    \begin{align*}
        \delta_t(T, F^N T) \leq \sum_{k = 0}^{l-1} \frac{N!}{(N-k)!} \delta_t(T, X_{k})
    \end{align*}
    This upper bound is a polynomial of degree $l-1$ in $N$.
    Therefore,
    \begin{align*}
        h_t(F) \leq \lim_{N\to\infty} \frac{1}{N}\log\left(\sum_{k = 0}^{l-1} \frac{N!}{(N-k)!} \delta_t(T, X_{k})\right) = 0\,,
    \end{align*}
    and, if $h_0(F)=0$, then
    \begin{align*}
        h^{\textup{pol}}(F) = \liminf_{N\to \infty} \log_N(\delta_0(T,F^N T)) \leq l-1\,.
    \end{align*}
\end{proof}

\begin{lemma}\label{lem:entropy_2}
    Let $M\in \lperf(A^e)$ be $(l,0,1)$-cyclotomic for some $l\geq 0$.
    \begin{enumerate}[label = (\roman*)]
    \item There exist $X_0,\dots,X_l \in \lperf(A^e)$ and for each $0\leq i < l$ a triangle
    \begin{align*}
        \eta_i\colon \hspace{1cm}
        X_i \xrightarrow{} M\lotimes_A X_i \xrightarrow{} X_{i+1} \xrightarrow{} X_i[1]
    \end{align*}
    such that $X_0 = A$, and $X_{l} = 0$.
    \item There exist $X_0',\dots,X_l' \in \lperf(A^e)$ and for each $0\leq i < l$ a triangle
    \begin{align*}
        \eta_i'\colon \hspace{1cm}
        X_{i+1}' \xrightarrow{} \tilde{M}\lotimes_A X_i' \xrightarrow{} X_{i}' \xrightarrow{} X_{i+1}'[1]
    \end{align*}
    such that $X_0' = A$, and $X_l' = 0$.
\end{enumerate}
\end{lemma}
\begin{proof}
    Let ${\bf f}=(f_l,\dots,f_1)\in\Hom_{A^e}(A,M)$ with $\Cone({\bf f})\simeq 0$.
    Set $X_0 := A$ and
    \begin{align*}
        X_i := \Cone(f_i) \lotimes_A \cdots \lotimes_A \Cone(f_1) 
    \end{align*}
    for $1\leq i \leq l$.
    Note $X_l = \Cone({\bf f}) \simeq 0$.
    For each $0 \leq i< l$, 
    apply the exact functor $-\lotimes_A X_i$ to the triangle
    \begin{align*}
        A \xrightarrow{f_{i+1}} M \xrightarrow{} \Cone(f_{i+1}) \xrightarrow{} A[1]
    \end{align*}
    to obtain the triangle $\eta_i$.
    This shows the first part.
    For the second part, set 
    \begin{align*}
        X_i' := (\tilde{M}[-1])^{i}\lotimes_A X_i\,,
    \end{align*}
    for each $0\leq i \leq l$.
    Note $X_0' \simeq A$ and $X_l' \simeq 0$.
    For each $0\leq i < l$,
    consider the rotated triangle
    \begin{align*}
        \eta_i^{rot}\colon \hspace{1cm}
        X_{i+1}[-1] \xrightarrow{} X_i \xrightarrow{} M\lotimes_A X_i \xrightarrow{} X_{i+1}\,,
    \end{align*}
    and apply the exact functor $\tilde{M} \lotimes_A (\tilde{M}[-1])^{i}\lotimes_A -$ 
    to obtain the triangle $\eta_i'$.
\end{proof}

\begin{proposition}\label{prop:cyclotomic_entropy}
    Let $l,m,n\in \bZ$ with $l\geq 1$ and $n\neq 0$.
    Suppose $M\in \lperf(A^e)$ is $(l,m,n)$-cyclotomic. 
    Then, 
    \begin{align*}
        h^{\textup{pol}}(M\lotimes_A ?) \leq l-1
        && \text{and} &&
        h_t(M\lotimes_A ?) = -\frac{m}{n} t\,,
    \end{align*}
    for all $t\in\mathbb{R}$.
\end{proposition}
\begin{proof}
    Fix $t\in \mathbb{R}$.
    The dg bimodule $N := M^n[m]$ is by assumption $(l,0,1)$-cyclotomic.
    Since $N\lotimes_A \tilde{N} \simeq A \simeq \tilde{N}\lotimes_A N$, we have
    \begin{align*}
        h_t(A\lotimes_A ?) \leq h_t(N\lotimes_A ?) + h_t(\tilde{N}\lotimes_A ?) 
    \end{align*}
    Since $A$ is proper, we have $h_t(A\lotimes_A ?) = 0$.
    By Lemmas \ref{lem:entropy_1} and \ref{lem:entropy_2}, we have that both $h_t(N\lotimes_A ?)$ and $h_t(\tilde{N}\lotimes_A ?)$ are non-positive.
    Therefore, both must be zero.
    It follows that 
    \begin{align*}
        0 = h_t(N\lotimes_A ?) = h_t(M^n\lotimes_A ?) + m t = nh_t(M\lotimes_A ?) + m t\,.
    \end{align*}
    This proves the second claim.
    On the other hand,
    \begin{align*}
        h^{\textup{pol}}(M\lotimes) \leq h^{\textup{pol}}(M^n\lotimes) = h^{\textup{pol}}(N\lotimes) \leq l-1
    \end{align*}
    where we use \cite[Lemma 2.10]{FFO} for the first estimate,
    and Lemmas \ref{lem:entropy_1} and \ref{lem:entropy_2} for the final estimate.
\end{proof}

\section{Serre cyclotomic algebras}\label{sec:serre_cyclotomic}

As before, let $A$ be a proper dg $K$-algebra.
Consider the dualizing dg bimodule $$\mathcal{N}_A := DA =\Hom_K(A,K) \in \Der(A^e)\,.$$
We call $A$ \emph{Iwanaga--Gorenstein} if $\mathcal{N}_A\in \lperf(A^e)$ and $\mathcal{N}_A$ is invertible.
In this case, the Nakayama functor
\begin{align*}
    \boldsymbol{\nu}_A := \mathcal{N}_A\lotimes_A - \colon \perf(A) \xrightarrow{\sim} \perf(A)
\end{align*}
is a Serre functor on $\perf(A)$.
The \emph{Coxeter transformation} of $A$ is the map $$\Phi_A\colon K_0(A) \to K_0(A)$$ induced by the Auslander--Reiten translation $\boldsymbol{\tau}_A := \boldsymbol{\nu}_A[-1]$.

\begin{definition}\label{def:serre_cyclotomic}
    Let $l,m,n\in\bZ$ with $l\geq 1$ and $n\neq 0$.
    We call a proper dg algebra $A$ \emph{Serre cyclotomic} of \emph{dimension} $(m,n)$ and \emph{order} at most $l$
    if $A$ is Iwanaga--Gorenstein and the inverse dualizing dg bimodule $\tilde{\mathcal{N}}_A$ is $(l,m,n)$-cyclotomic.
    We sometimes say that $A$ is $(l,m,n)$-Serre cyclotomic.
\end{definition}

Recall that an Iwanaga--Gorenstein dg algebra $A$ is called \emph{fractionally Calabi--Yau} of \emph{dimension} $(m,n)$
if there exists an isomorphism $\mathcal{N}_A^n \xrightarrow{\sim} A[m]$ in $\Der(A^e)$.
Note that $A$ is $(1,m,n)$-Serre cyclotomic if and only if $A$ is fractionally Calabi--Yau of dimension $(m,n)$.

\begin{corollary}\label{cor:serre_cyclotomic_cyclotomic_type}
    Let $l,m,n\in \bZ$ with $l\geq 1$ and $n\neq 0$. 
    Suppose $A$ is $(l,m,n)$-Serre cyclotomic. 
    Then, the Coxeter transformation $\Phi_A$ is cyclotomic with
    \begin{align*}
        (\Phi_A^{2n}-\bm{1})^l = 0\,.
    \end{align*}
\end{corollary}
\begin{proof}
    Since $\Phi_A = \Psi_{\mathcal{N}_A[-1]} = - \Psi_{\tilde{\mathcal{N}}_A}^{-1}$,
    this follows immediately from Proposition \ref{prop:induced_map_cyclotomic} with $M:=\tilde{\mathcal{N}}_A$.
\end{proof}

\begin{corollary}\label{cor:serre_cyclotomic_entropy}
    Let $l,m,n\in\mathbb{Z}$ with $l\geq 1$ and $n\neq 0$.
    Suppose $A$ is $(l,m,n)$-Serre cyclotomic.
    Then the entropy and polynomial entropy of the Serre functor satisfy 
    \begin{align*}
        h^{\textup{pol}}(\boldsymbol{\nu}_A)\leq l-1 && \text{and} && h_t(\boldsymbol{\nu}_A) = \frac{m}{n}\cdot t
    \end{align*}
    for all $t\in\mathbb{R}$.
    In particular, the fraction $\frac{m}{n}\in \mathbb{Q}$ is a derived invariant. 
\end{corollary}
\begin{proof}
    The properties of the entropy and polynomial entropy follow from the proof of Proposition \ref{prop:cyclotomic_entropy} with $\tilde{M} = \mathcal{N}_A$.
    It follows that the fraction $\frac{m}{n}$ is the Serre dimension of \cite{EL} which is a derived invariant because Serre functors are unique up to natural isomorphism.
\end{proof}

\begin{lemma}\label{lem:serre_cyclotomic_derived_invariant}
    Let $l,m,n\in \bZ$ with $l\geq 1$ and $n\neq 0$.
    Suppose that $A$ and $B$ are derived Morita equivalent Iwanaga--Gorenstein dg algebras.
    Then, $A$ is $(l,m,n)$-Serre cyclotomic if and only if $B$ is $(l,m,n)$-Serre cyclotomic.
\end{lemma}
\begin{proof}
    Set $C := A\lotimes_K B^\op$ and identify $C^\op \cong B \lotimes_K A^\op$.
    By assumption there exists a dg $C$-module $M$ and a dg $C^\op$-module $\tilde{M}$
    such that 
    \begin{equation*}
        \begin{tikzcd}
            M \lotimes_A ? \colon \Der(A) \ar[r,shift left = 1mm] & \Der(B) \,\colon\! \tilde{M}\lotimes_B ? \ar[l,shift left = 1mm] \,
        \end{tikzcd}
    \end{equation*}
    are well-defined mutually quasi-inverse equivalences.
    With $L := M\otimes_K \tilde{M}$ and $\tilde{L} := \tilde{M}\otimes_K M$ we obtain equivalences
    \begin{equation*}
    \begin{tikzcd}
        L \lotimes_{A^e} ? \colon \Der(A^e) \ar[r,shift left = 1mm] & \Der(B^e) \,\colon\! \tilde{L}\lotimes_{B^e} ? \ar[l,shift left = 1mm] \,.
    \end{tikzcd}
    \end{equation*}
    By the uniqueness of Serre functors,
    we find $\tilde{\mathcal{N}}_B \simeq L \lotimes_{A^e} \tilde{\mathcal{N}}_A$ hence $\tilde{\mathcal{N}}_B^n \simeq L \lotimes_{A^e} \tilde{\mathcal{N}}_A^n$.
    Assume that $A$ is $(l,m,n)$-Serre cyclotomic.
    Let ${\bf f}\in \Hom_{A^e}(A,\tilde{\mathcal{N}}_A^n[m])^l$ with $\Cone({\bf f}) \simeq 0$.
    Consider ${\bf f}' \in \Hom_{B^e}(B, \tilde{\mathcal{N}}_B^n[m])^l$ defined by 
    \begin{align*}
        f'_i := L \lotimes_{A^e} f_i
    \end{align*}
    for all $1\leq i\leq l$. 
    Then 
    \begin{align*}
        \Cone({\bf f}') &\simeq (L \lotimes_{A^e} \Cone(f_l)) \lotimes_{B} \dots \lotimes_{B} (L \lotimes_{A^e} \Cone(f_1)) \\
        & \simeq M\lotimes_A \Cone(f_l) \lotimes_A \tilde{M} \lotimes_{B} \dots \lotimes_{B} M\lotimes_A \Cone(f_1) \lotimes_A \tilde{M} \\
        & \simeq L \lotimes_{A^e} \Cone({\bf f}) \simeq 0
    \end{align*}
    where we use that $\tilde{M}\lotimes_B M \simeq A$ in $\Der(A^e)$.
\end{proof}

\subsection{(Higher) canonical algebras}\label{sec:canonical}

Let $t\geq 2$, ${\bf p} = (p_1,\dots,p_t) \in \mathbb{N}_{\geq 1}^t$ a weight sequence, 
and $\boldsymbol{\lambda} = (\lambda_1,\dots,\lambda_t)\in \bP^1(K)^t$ pairwise different points on the projective line.
We may assume without loss of generality that $\lambda_1 = \infty$, $\lambda_2 = 0$ and $\lambda_i\in K^*$ for all $i\geq 2$.
Ringel's canonical algebra $A:=A({\bf p},\boldsymbol{\lambda})$ is the path algebra of the quiver $Q({\bf p})$ given by
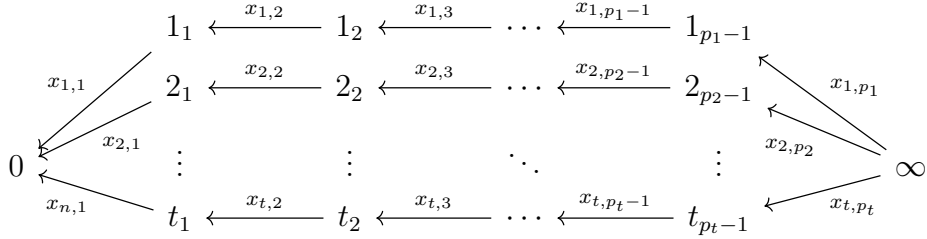
\[
\begin{tikzcd}[row sep = .1cm, column sep = 1.5cm] 
    & 1_1 \ar["x_{1,2}",r,<-] & 1_2 \ar["x_{1,3}",r,<-] & \cdots \ar["x_{1,p_1-1}",r,<-] & 1_{p_1-1} \ar["x_{1,p_1}",ddr,<-] \\
    & 2_1 \ar["x_{2,2}",r,<-] & 2_2 \ar["x_{2,3}",r,<-] & \cdots \ar["x_{2,p_2-1}",r,<-] & 2_{p_2-1} \ar["x_{2,p_2}"',dr,<-] \\
    0 \ar["x_{1,1}",uur,<-] \ar["x_{2,1}"',ur,<-] \ar["x_{n,1}"',dr,<-] & \vdots & \vdots & \ddots & \vdots & \infty \\
    & t_1 \ar["x_{t,2}",r,<-] & t_2 \ar["x_{t,3}",r,<-] & \cdots \ar["x_{t,p_t-1}",r,<-] & t_{p_t-1} \ar["x_{t,p_t}"',ur,<-] \\
\end{tikzcd}
\]
modulo the ideal $I({\bf p},{\boldsymbol{\lambda}})$ generated by the relations
\begin{align*}
    x_{i,1}x_{i,2}\cdots x_{i,p_i} = x_{2,1}\cdots x_{2,p_2} - \lambda_i x_{1,1}\cdots x_{1,p_1}
\end{align*}
for all $3\leq i\leq t$.
We consider $A$ as a dg algebra concentrated in degree $0$.
Geigle--Lenzing \cite{GL} proved that $\perf(A)$ is equivalent to the bounded derived category $\Der^b(\coh(\bX))$ of coherent sheaves on a weighted projective line $\bX := \bX({\bf p},\boldsymbol{\lambda})$.
Its Picard group $\bL := \bL({\bf p})$ is the abelian group generated by $\vec{x}_1,\dots,\vec{x}_t,\vec{c}$ subject to the relations $\vec{c} = p_i\vec{x}_i$ for all $1\leq i\leq t$.
Write $\mathcal{O}(\vec{y})$ for the line bundle corresponding to $\vec{y}\in \bL$.
In particular, for the neutral element $\vec{0}\in\bL$, is $\mathcal{O} := \mathcal{O}(\vec{0})$ the structure sheaf of $\bX$.
A Serre functor on $\Der^b(\coh(\bX))$ is given by $\mathcal{O}(\vec{\omega})[1]\lotimes_{\bX}?$
where $\mathcal{O(\vec{\omega})}$ is the dualizing sheaf corresponding to
\begin{align*}
    \vec{\omega} := (t-2)\vec{c} - \sum_{i=1}^t \vec{x}_i\,.
\end{align*}

\begin{proposition}\label{prop:canonical}
    Let $t\geq 2$, 
    ${\bf p} \in \mathbb{N}^t$ and 
    $\boldsymbol{\lambda} \in \bP^1(K)$ as above.
    Set 
    \begin{align*}
        p:=\lcm(p_1,\dots,p_t)
        &&\text{and}&&
        \delta := (t-2) p - \sum_{i=1}^t \frac{p}{p_i}\,.
    \end{align*}
        \text{The canonical algebra $A({\bf p},\boldsymbol{\lambda})$ is }
        $\begin{cases}
            \text{$(2,p,p)$-Serre cyclotomic} & \text{if $\delta <0$;}\\
            \text{fractionally $(p,p)$-Calabi--Yau} & \text{if $\delta =0$;}\\
            \text{$(2,-p,-p)$-Serre cyclotomic} & \text{if $\delta >0$.}
        \end{cases}$
\end{proposition}
\begin{proof}
    Let $T\in \coh(\bX)$ be the canonical tilting bundle with $\End_{\bX}(T)^\op \cong A$.
    We have the following, up to natural isomorphism, commutative diagrams 
    \begin{equation*}
        \begin{tikzcd}[column sep = 2cm]
            \perf(A) \ar[r, shift left = 1mm, "T\lotimes_A?"] \ar[d,"A\lotimes_A?"'] & \Der^b(\coh(\bX)) \ar[l, shift left = 1mm, "\text{$\rHom_\bX(T,?)$}"] \ar[d,"\mathcal{O}\lotimes_{\bX}?"]\\
            \perf(A) \ar[r, shift left = 1mm, "T\lotimes_A?"] & \Der^b(\coh(\bX)) \ar[l, shift left = 1mm, "\text{$\rHom_\bX(T,?)$}"] 
        \end{tikzcd}
        \hspace{1cm}
        \begin{tikzcd}[column sep = 2cm]
            \perf(A) \ar[r, shift left = 1mm, "T\lotimes_A?"] \ar[d,"\text{$\mathcal{N}_A[-1]$}\lotimes_A?"'] & \Der^b(\coh(\bX)) \ar[l, shift left = 1mm, "\text{$\rHom_\bX(T,?)$}"] \ar[d,"\mathcal{O}(\vec{\omega})\lotimes_{\bX}?"]\\
            \perf(A) \ar[r, shift left = 1mm, "T\lotimes_A?"] & \Der^b(\coh(\bX)) \ar[l, shift left = 1mm, "\text{$\rHom_\bX(T,?)$}"] 
        \end{tikzcd}
    \end{equation*}
    where the horizontal functors are mutually quasi-inverse equivalences.
    If $\delta = 0$, it is well-known that $A$ is fractionally Calabi--Yau of dimension $(p,p)$.
    Suppose $\delta\neq 0$.
    Let $\varepsilon=\pm 1$ such that $\varepsilon\delta = \lrvert{\delta} > 0$.
    Note that $\vec{y} := \varepsilon p \vec{\omega} = \lrvert{\delta}\vec{c}$ in $\bL$. 
    It follows that
    \begin{align*}
        \Hom_{A^e}(A,(\mathcal{N}_A[-1])^{\varepsilon p}) \simeq \Hom_{\bX}(\mathcal{O},\mathcal{O}(\vec{y})) \simeq K[T_1,T_2]_{\lrvert{\delta}}\,,
    \end{align*}
    where the latter denotes the degree $\lrvert{\delta}$ homogeneous part of the polynomial ring in two variables.
    Let $f_1,f_2\in \Hom_{\bX}(\mathcal{O},\mathcal{O}(\vec{y}))$ be the two homomorphisms corresponding to $T_1^{\lrvert{\delta}},T_2^{\lrvert{\delta}} \in K[T_1,T_2]_{\lrvert{\delta}}$, respectively.
    In $\Der^b(\coh(\bX))$, the derived tensor product $\Cone(f_2)\lotimes_{\bX}\Cone(f_1)$ is isomorphic to the complex
    \begin{align*}
        0 \to \mathcal{O} \xrightarrow{\bsm{f_1\\f_2}} \mathcal{O}(\vec{y})\oplus \mathcal{O}(\vec{y}) \xrightarrow{\bsm{-f_2 & f_1}} \mathcal{O}(2\vec{y}) \to 0\,.
    \end{align*}
    This is the Koszul complex of the common zero locus of the sections $f_1$ and $f_2$. 
    Its zeroth cohomology vanishes in $\coh(\bX)$. It follows that $\Cone(f_2)\lotimes_{\bX}\Cone(f_1)$ is acyclic.
\end{proof}

\begin{remark}
    Along the same lines one shows that $d$-canonical algebras from \cite{HIMO} are Serre cyclotomic of order at most $d+1$.
\end{remark}

\subsection{Graded gentle algebras}\label{sec:graded_gentle}

In this section we assume that $A = KQ/I$
is a graded gentle algebra given by a finite connected quiver $Q$ and an ideal $I$ 
and the grading is given by a function $\lrvert{\cdot}\colon Q_1 \to \bZ$
on the set of arrows $Q_1$ of $Q$.
We consider $A$ as a dg algebra with zero differential.
Recall that a dg algebra $A$ is called homologically smooth if $A$ is perfect as a dg $A$-bimodule.

\begin{proposition}\label{prop:graded_gentle}
    Let $A$ be a homologically smooth and proper connected graded gentle algebra.
    Then $A$ is Serre cyclotomic if and only if $A$ is derived equivalent to a hereditary algebra of type $\mathbb{A}$ or $\tilde{\mathbb{A}}$.
\end{proposition}
\begin{proof}
    Any hereditary algebra of finite type $\mathbb{A}$ is fractionally Calabi--Yau; see \cite{MY}.
    Any hereditary algebra of affine type $\tilde{\mathbb{A}}$ is derived equivalent to a canonical algebra with $t = 2$
    and therefore Serre cyclotomic by Proposition \ref{prop:canonical}.
    Therefore, the ``if'' part follows with Lemma \ref{lem:serre_cyclotomic_derived_invariant}.

    For the ``only if'' part we use recent work of Chang--Elagin--Schroll \cite{CS}.
    They compute the entropy of the Serre functor: 
    Suppose $A$ is not derived equivalent to a hereditary algebra of finite type $\mathbb{A}$.
    Then, there exists a finite subset $\Omega\subseteq \mathbb{Q}$ such that
    \begin{align*}
        h_t(\boldsymbol{\nu}_A) = 
        \begin{cases}
            (1-\min\Omega)\cdot t, & t\geq 0; \\
            (1-\max\Omega)\cdot t, & t\leq 0.
        \end{cases}
    \end{align*}
    Furthermore, $\min\Omega = \max\Omega$ only if $A$ is derived equivalent to a hereditary algebra of affine type $\tilde{\mathbb{A}}$ concentrated in degree $0$.
    According to Corollary \ref{cor:serre_cyclotomic_entropy}, $A$ cannot be Serre cyclotomic unless $\min\Omega = \max\Omega$.
\end{proof}

\begin{example}
    Consider the algebra $A = K Q/I$
    given by the quiver 
    \[
    \begin{tikzcd}
        Q \colon & 1 \ar[loop left, "b_1"] \ar[<-,r,"a"] & 2 \ar[loop right, "b_2"]
    \end{tikzcd}
    \]
    modulo the ideal $I := \lrangle{b_1^2,b_2^2}$.
    This is a gentle algebra of infinite global dimension,
    but still Iwanaga--Gorenstein.
    A lengthy direct computation shows that $A$ is $(2,2,2)$-Serre cyclotomic.
\end{example}

\subsection{Graded path algebras}\label{sec:graded_quiver}
In this section we consider graded path algebras $A$,
that is $A$ is given by the path algebra $KQ$ of a finite connected quiver $Q$ 
and graded by a map $\lrvert{\cdot}\colon Q_1 \to \bZ$.
We consider $A$ as a dg algebra with zero differential.

\begin{proposition}\label{prop:graded_quiver}
    Let $A$ be a connected graded path algebra.
    Then $A$ is Serre cyclotomic if and only if $A$ is derived equivalent to a path algebra of a quiver of finite or affine type concentrated in degree $0$.
\end{proposition}
\begin{proof}
    Let $A=(KQ,\lrvert{\cdot})$ for a finite quiver $Q$ and a grading $\lrvert{\cdot}\colon Q_1 \to \mathbb{Z}$.
    Denote by $B=KQ$ the path algebra of $Q$ concentrated in degree $0$.
    Note that $A$ is derived equivalent to $B$ if $Q$ is a tree.
    If $Q$ is of finite type, then $B$ is fractionally Calabi--Yau, see \cite{MY}.
    If $Q$ is of affine type, then $B$ is derived equivalent to a canonical algebra with $\delta < 0$.
    It follows from Lemma \ref{lem:serre_cyclotomic_derived_invariant} and Proposition \ref{prop:canonical} that $B$ is Serre cyclotomic in this case.

    Suppose $Q$ is of infinite type.
    Then $\boldsymbol{\nu}_B^{N}(DB)[-N] \cong \tau_B^{N}(DB) \in \mod(B)$ is the $N$-th Auslander--Reiten translation of $DB$ for all $N\geq 0$. 
    Therefore, the entropy of the Nakayama functor $\boldsymbol{\nu}_A$ at $t=0$ 
    can be computed with \cite[Theorem 2.6]{DHKK}:
    \begin{align*}
        h_0(\boldsymbol{\nu}_A) 
        = \lim_{N\to\infty} \frac{1}{N}\log \sum_{k\in \bZ} \dim H^{k}(\boldsymbol{\nu}_A^N(DA)) 
        = \lim_{N\to\infty} \frac{1}{N}\log \dim \tau_B^N(DB) 
        = \log \rho(\Phi_{B}),
    \end{align*}
    where $\rho(\Phi_{B})$ is the spectral radius of the Coxeter matrix $\Phi_{B}$ of the path algebra $B$.
    It is well known that $\rho(\Phi_{B}) > 1$ if $Q$ is of indefinite type.
    In view of Corollary \ref{cor:serre_cyclotomic_entropy}, $A$ cannot be Serre cyclotomic if $Q$ is of indefinite type.

    The only remaining case is when $Q$ is of affine type $\tilde{\mathbb{A}}$.
    In this case, $A$ is a graded gentle algebra and the claim follows from Proposition \ref{prop:graded_gentle}.
\end{proof}

\subsection{Complexity of trivial extension algebras}\label{sec:complexity}

Let $A$ be a finite-dimensional algebra.
Following Alperin \cite{Al},
the \emph{complexity} of a module $V\in \mod(A)$ is defined as 
\begin{align*}
    \textup{cx}_{A}(V) := \inf\left\{k \geq 0 \mid \text{$\exists C\geq 0\, \forall n\gg 0 \colon \dim P_n \leq C n^{k-1}$}\right\} \in \mathbb{Z}_{\geq 0} \cup \{\infty\}\,
\end{align*}
where $\cdots \to P_n \to \cdots \to P_1 \to P_0 \twoheadrightarrow V$
is a minimal projective resolution of $V$ in $\mod(A)$.
We define the \emph{global complexity} of $A$ as 
\begin{align*}
    \glcx(A) := \sup\{\cx_A(V) \mid V\in\mod(A)\} \in \mathbb{Z}_{\geq 0} \cup\infty.
\end{align*}
For example, $\glcx(A) = 0$ if and only if $A$ has finite global dimension.
The trichotomy of quivers $Q$ is classically known to be refelected by the global complexity of the trivial extension algebras $TA = A \ltimes DA$ of their path algebras $A=KQ$; see e.g. \cite{Pu}:
A quiver $Q$ is of finite type if $\glcx(TA)\leq 1$, of affine type if $\glcx(TA)=2$, and of indefinite type if $\glcx(TA) = \infty$. 
Chan--Darpö--Iyama-Marczinzik proved in \cite{CDIM} that the trivial extension algebra $TA$ 
of an algebra $A$ is a periodic selfinjective algebra if and only if $A$ is fractionally Calabi--Yau and of finite global dimension.
In particular, $\glcx(A) \leq 1$ for such algebras.

Suppose from now on that $A$ is a finite-dimensional algebra of finite global dimension.
The connection between the dynamics of the Serre functor on $\perf(A)$ and the growth of projective resolution in $\mod(TA)$ is given by Happel's equivalence, see \cite{Ha}:
The trivial extension algebra $TA = A \ltimes DA$ is $\mathbb{Z}$-graded with $A$ in degree $0$ and $DA$ in degree $1$.
We may canonically identify the category of finite-dimensional graded $TA$-modules
with the category $\mod(RA)$ of finite-dimensional modules over the repetitive algebra $RA$ of $A$.
Via this identification, we have a forgetful functor $F\colon \mod(RA) \xrightarrow{} \mod(TA)$,
and a grade shift functor $(1)\colon \mod(RA) \xrightarrow{} \mod(RA)$.
The category $\mod(RA)$ is Frobenius exact,
hence its stable category $\underline{\mod}(RA)$ is triangulated
with suspension given by Heller's cosyzygy functor $\Omega_{RA}^{-1}$.
Happel's equivalence $H\colon \perf(A) \xrightarrow{\sim} \underline{\mod}(RA)$ is constructed in \cite{Ha},
and it makes the following square commutative:
\[
\begin{tikzcd}
\perf(A) \ar[rr,"H"] \ar[d,"\boldsymbol{\nu}_A"'] && \underline{\mod}(RA) \ar[d,"\Omega_{RA} (1)"] \\
\perf(A) \ar[rr,"H"] && \underline{\mod}(RA) 
\end{tikzcd}
\]
We obtain the following result for Serre cyclotomic algebras:

\begin{proposition}\label{prop:trivial_extension_complexity}
    Let $A$ be a finite-dimensional algebra of finite global dimension,
    and $l\geq 1$.
    Suppose that $A$ is Serre cyclotomic of order at most $l$. 
    Then, the global complexity of its trivial extension algebra $TA = A \ltimes DA$ is finite with 
    $$\glcx(TA) \leq l.$$
\end{proposition}
\begin{proof}
    Suppose $A$ is Serre cyclotomic of dimension $(m,n)$ with $d:=n+m\geq 0$.
    The case $d\leq 0$ can be proved similarly.
    Let $S$ be a simple $TA$-module.
    Evidently, $S$ is gradable, that is there exists a simple $RA$-module $S'$ with $F(S') = S$.
    Let $S''\in\perf(A)$ with $S' = H(S'')$.
    The triangles $\eta_i'\lotimes_A S''$ in $\perf(A)$ from Lemma \ref{lem:entropy_2} correspond, via Happel's equivalence, to to short exact sequences in $\mod(RA)$:
    \begin{align*}
        0\to V_{i+1} \to \Omega_{RA}^{d}V_i(n) \to V_{i} \to 0
    \end{align*}
    with $V_0 \cong S'$ and $V_l \cong 0$ for $0\leq i < l$.
    Since $\Hom_{RA}(S'(n),S') = 0$ for all $n\neq 0$, we have $d\neq 0$.
    Let 
    \[
    \begin{tikzcd}[column sep = 0.2cm]
        P_\bullet^{(i)}\colon \cdots \ar[rr] && P_k^{(i)} \ar[rr] \ar[dr] && \cdots \ar[rr] \ar[dr] && P_1^{(i)} \ar[rr] \ar[dr] && P_0^{(i)} \ar[dr,twoheadrightarrow] \\
        & \cdots \ar[ur] && \Omega_{RA}^k(V_i) \ar[ur] & \cdots & \Omega_{RA}^2(V_i) \ar[ur] && \Omega_{RA}(V_i) \ar[ur] && V_i
    \end{tikzcd}
    \]
    be a minimal projective resolution of $V_i$ in $\mod(RA)$ for $1\leq i \leq l$.
    Applying the forgetful functor $F$, yields a minimal projective resolution of $F(V_i)$ in $\mod(TA)$ for each $0\leq i < l$.
    The Horseshoe Lemma applied to the short exact sequences above gives the upper bounds 
    \begin{align*}
        \dim P_{kd+r}^{(i)} &\leq \dim P_{(k-1)d+r}^{(i)} + \dim P_{(k-1)d+r}^{(i+1)}
    \end{align*}
    for all $0\leq i < l$, $k\geq 0$, and $0\leq r < d$.
    Inductively, for $k\geq l$ and $0\leq r < d$, we get 
    \begin{align*}
        \dim P_{kd+r}^{(0)} 
        & \leq \sum_{i=0}^{l-1} \binom{k}{i} \dim P_{r}^{(i)} 
    \end{align*}
    which is a polynomial expression of degree $l-1$ in $k$.
    This proves the claim.
\end{proof}


\section*{Acknowledgements}

I am deeply grateful to Sibylle Schroll for many stimulating discussions on the subject.
I would like to thank Erlend D. Børve, Xiaofa Chen, Marvin Plogmann and Jan Thomm for patiently answering questions on triangulated categories and differential graded algebras.


\bibliographystyle{amsalpha}

\end{document}